\documentclass[10pt,a4paper]{article}
\usepackage{amsfonts}
\usepackage{amssymb}
\usepackage{amsmath}
\usepackage{amsthm}
\usepackage[english]{babel}

\newtheorem{theorem}{Theorem}[section]

\newtheorem{lemma}[theorem]{Lemma}
\newtheorem{corollary}[theorem]{Corollary}

\theoremstyle{definition}

\theoremstyle{remark}

\numberwithin{equation}{section}

\DeclareMathOperator{\diver}{\mathrm{div}}
\DeclareMathOperator{\dist}{\mathrm{dist}}

\newcommand{\RR}{\mathbb{R}}
\newcommand{\Om}{\Omega}
\newcommand{\si}{\sigma}
\newcommand{\pa}{\partial}
\newcommand{\ep}{\varepsilon}

\def\bysame{\leavevmode\hbox to3em{\hrulefill}\thinspace}

\begin{document}

\title{A weak comparison principle for solutions of very degenerate elliptic equations}

\author{Giulio Ciraolo\thanks{Department of Mathematics and Informatics, Universit\`a di Palermo, Via Archirafi 34, 90123, Italy. E-mail: {\tt g.ciraolo@math.unipa.it }}}


\date{\today}

\maketitle

\begin{abstract}
We prove a comparison principle for weak solutions of elliptic quasilinear equations in divergence form whose ellipticity constants degenerate at every point where $\nabla u \in K$, where $K\subset \RR^N$ is a Borel set containing the origin.
\end{abstract}

\section{Introduction}
Let $K\subset \RR^N$, $N\geq 2$, be a Borel set containing the origin $O$. We consider a vector function $A: \RR^N \to \RR^N$, $A \in L_{loc}^\infty(\RR^N)$, such that
\begin{equation}\label{A def}
\begin{cases}
A (\xi) = 0, &  \textmd{ if } \xi \in K,  \\
[A(\xi) - A(\eta)]\cdot (\xi - \eta) > 0 ,\ \ \forall\: \eta \in \RR^N \setminus \{\xi\}, &  \textmd{ if } \xi \not\in K,
\end{cases}
\end{equation}
where $\cdot$ denotes the scalar product in $\RR^N$. In this note we prove a comparison principle for Lipschitz weak solutions of
\begin{equation} \label{Euler eq 1}
\begin{cases}
- \diver A(\nabla u) = g, & \textmd{in } \Om,\\
u = \psi, & \textmd{on } \pa \Om,
\end{cases}
\end{equation}
where $\Om$ is a bounded domain in $\RR^N$, $\psi \in W^{1,\infty}(\Om)$ and $g\in L^1(\Omega)$.
As usual, $u \in W^{1,\infty}(\Om)$ is a weak solution of \eqref{Euler eq 1} if $u-\psi \in W_0^{1,\infty}(\Om)$ and $u$ satisfies
\begin{equation}\label{Euler eq 2}
\int_\Omega A(\nabla u) \cdot \nabla \phi dx = \int_\Omega g \phi dx, \quad \textmd{for every } \phi\in C_0^1(\Omega).
\end{equation}
For weak comparison principle  we mean the following: if $u_1,u_2$ are two solutions of \eqref{Euler eq 2} with $u_1\leq u_2$ on $\pa \Om$, then $u_1\leq u_2$ in $\overline{\Om}$. Clearly, the weak comparison principle implies the uniqueness of the solution.

It is well known that if $K$ is the singleton $\{O\}$, then \eqref{A def} guarantees the validity of the weak comparison principle (see for instance \cite{GT} and \cite{To}). For this reason, from now on $K$ will be a set containing the origin and at least another point of $\RR^N$.

Our interest in this kind of equations comes from recent studies in traffic congestion problems (see \cite{BC} and \cite{BCS}), complex-valued solutions of the {\it eikonal} equation (see \cite{MT1}--\cite{MT4}) and in variational problems which are relaxations of non-convex ones (see for instance \cite{CarMul} and \cite{FFM}).

As an example, we can think to $f:[0,+\infty) \to [0,+\infty)$ given by
\begin{equation}\label{f examp 1}
f(s)=\frac{1}{p}(s-1)_+^p,
\end{equation}
where $p>1$ and $(\cdot)_+$ stands for the positive part, and consider the functional
\begin{equation}\label{I(u)}
I(u)= \int_\Om [f(|\nabla u(x)|)-g(x) u(x)]dx, \quad u\in \psi + W_0^{1,\infty}(\Om).
\end{equation}
As it is well-known, \eqref{Euler eq 2} is the Euler-Lagrange equation associated to \eqref{I(u)} with $A$ given by
\begin{equation}\label{A grad f}
A(\nabla u) = \frac{f'(|\nabla u|)}{|\nabla u|} \nabla u,
\end{equation}
and it is easy to verify that $A$ satisfies \eqref{A def} with $K=\{ \xi \in \RR^N:\ |\xi|\leq 1\}$. It is clear that in this case the monotonicity condition in \eqref{A def} can be read in terms of the convexity of $f$. Indeed, $f$ is not strictly convex in $[0,+\infty)$ since it vanishes in $[0,1]$; however, if $s_1>1$ then
\begin{equation*}
f((1-t)s_0+ts_1) < (1-t)f(s_0)+tf(s_1),  \quad t\in [0,1],
\end{equation*}
for any  $s_0 \in [0,+\infty)$ and $s_0 \neq s_1$: the convexity holds in the strict sense whenever a value greater than $1$ is considered.

Coming back to our original problem we notice that, since $A$ vanishes in $K$, \eqref{Euler eq 1} is strongly degenerate and no more than Lipschitz regularity of the solution can be expected. It is clear that if $g=0$, then every function with gradient in $K$ will satisfy the equation. Besides the papers cited before, we mention \cite{Br,CCG,EMT,SV} where regularity issues were tackled and \cite{Ci} where it is proven that solutions to \eqref{Euler eq 1} satisfy an obstacle problem for the gradient in the viscosity sense. Here, we will not specify the assumptions on $A$ and $g$ that guarantee the existence of a Lipschitz solution and we refer to the mentioned papers for this interesting issue.

We stress that some regularity may be expected if we look at $A(\nabla u)$. In \cite{BCS} and \cite{CarMul} the authors prove some Sobolev regularity results for $A(\nabla u)$ under more restrictive assumptions on $A$ and $g$. We also mention that results on the continuity of $A(\nabla u)$ can be found in \cite{CF} and \cite{SV}.

In Section \ref{section 2}, we prove a weak comparison principle for Lipschitz solutions of \eqref{Euler eq 2} by assuming the following: (i) one of the two solutions satisfies a Sobolev regularity assumption on $A(\nabla u)$; (ii) the Lebesgue measure of the set where $g$ vanishes is zero. As we shall prove, the former guarantees that the set where $\nabla u \in K$ and $g$ does not vanish has measure zero. The latter seems to be optimal for proving our result. Indeed, if we assume that $g=0$, then any Lipschitz function with gradient in $K$ would be a solution and we can not have a comparison between any two of such solutions. For instance, if we consider $A$ as in \eqref{A grad f} with $f$ given by \eqref{f examp 1}, then a simple example of functions that satisfy \eqref{Euler eq 1} is given by $u_\si(x)=\si \dist(x,\pa \Om)$, with $\si \in [-1,1]$. Since every $u_\si=0$ on $\pa \Om$, \eqref{Euler eq 1} does not have a unique solution and a comparison principle can not hold. Generally speaking, any region where $g$ vanishes will be source of problems for proving a comparison principle. We mention that, for $A$ as in \eqref{A grad f} and $g=1$, a comparison principle for minimizers of \eqref{I(u)} was proven in \cite{CMS}.

\section{Main result} \label{section 2}
Before proving our main result, we need the following lemma which generalizes a result obtained in \cite{Lo} for the p-Laplacian. In what follows, $|D|$ denotes the Lebesgue measure of a set $D \subset \RR^N$.

\begin{lemma} \label{lemma 1}
Let $u \in W^{1,\infty}(\Omega)$ be a solution of \eqref{Euler eq 2}, with $A$ satisfying \eqref{A def} and let
\begin{equation}\label{Z def}
Z=\{x\in \Omega: \ \nabla u(x)  \in K \}.
\end{equation}
If $A(\nabla u) \in W^{1,p}(\Om)$ for some $p\geq 1$, then
\begin{equation} \label{Z minus G0}
|Z \setminus G_0|=0,
\end{equation}
where
\begin{equation}\label{G0}
G_0=\{x\in \Om:\ g(x)=0\}.
\end{equation}
In particular, if $|G_0| =0$ then $|Z|=0$.
\end{lemma}

\begin{proof}
Since $A(\nabla u) \in W^{1,p}(\Om)$, then the function
\begin{equation*}
\frac{|A(\nabla u)|}{\ep + |A(\nabla u)|} \in W^{1,p}(\Om),
\end{equation*}
for any $\ep >0$. Let $\psi \in C_0^1 (\Omega) $, set
\begin{equation*}
\phi(x)=\frac{|A(\nabla u(x))|}{\ep + |A(\nabla u(x))|} \psi (x),
\end{equation*}
and notice that $\phi \in L^{\infty}(\Omega) \cap W_0^{1,p}(\Omega)$. Since $u$ is Lipschitz continuous and $A \in L^{\infty}_{loc}(\RR^N)$, we have that $A(\nabla u) \in L^{\infty}(\Omega)$. Hence, by an approximation argument, $\phi$ can be used as a test function in \eqref{Euler eq 2}, yielding
\begin{multline} \label{eq1 in prop}
\int_\Omega \frac{|A(\nabla u)|}{\ep + |A(\nabla u)|} A(\nabla u) \cdot \nabla \psi dx + \ep \int_\Omega  \psi \frac{A(\nabla u) \cdot \nabla |A(\nabla u)|}{(\ep + |A(\nabla u)|)^2} dx = \\ = \int_\Omega \frac{|A(\nabla u)|}{\ep + |A(\nabla u)|} \psi g dx.
\end{multline}
It is clear that
\begin{equation} \label{eq2 in prop}
\int_\Omega \frac{|A(\nabla u)|}{\ep + |A(\nabla u)|} \psi g dx = \int_{\Omega \setminus Z} \frac{|A(\nabla u)|}{\ep + |A(\nabla u)|} \psi g dx,
\end{equation}
and that Cauchy-Schwarz inequality yields
\begin{equation} \label{eq3 in prop}
\Big{|}\ep \frac{A(\nabla u) \cdot \nabla |A(\nabla u)|}{(\ep + |A(\nabla u)|)^2} \Big{|} \leq |\nabla ( |A(\nabla u)|) |
\end{equation}
uniformly for $\ep>0$. Since $\nabla ( |A(\nabla u)|) \in L^p(\Om)$, from \eqref{eq1 in prop}--\eqref{eq3 in prop} and by letting $\ep$ go to zero, we obtain from Lebesgue's dominated convergence Theorem that
\begin{equation*}
\int_\Omega A(\nabla u) \cdot \nabla \psi dx = \int_{\Omega \setminus Z} g \psi dx,
\end{equation*}
for any $\psi \in C_0^1(\Omega)$. From \eqref{Euler eq 2} we have
\begin{equation*}
\int_\Omega g \psi dx = \int_{\Omega\setminus Z} g \psi dx \quad \textmd{ for any } \psi \in C_0^1(\Omega),
\end{equation*}
that is
\begin{equation*}
g(x)=0  \  \textmd{ for almost every }\  x \in Z,
\end{equation*}
which implies \eqref{Z minus G0}. \qed
\end{proof}

Our main result is the following.

\begin{theorem}\label{th comp}
Let $u_j \in W^{1,\infty}(\Omega),\: j=1,2,$ be two solutions of \eqref{Euler eq 2}, with $A$ satisfying \eqref{A def} and $g$ such that $|G_0|=0$, with $G_0$ given by \eqref{G0}. Furthermore, let us assume that $A(\nabla u_j) \in W^{1,p}(\Omega)$ for some $p\geq 1$ and $j \in \{1,2\}$.

If $u_1\leq u_2$ on $\pa \Om$ then $u_1 \leq u_2$ in $\overline{\Om}$.
\end{theorem}

\begin{proof}
We proceed by contradiction. Let us assume that $U=\{x \in \Om:\ u_1>u_2\}$ is nonempty. Since $u_1$ and $u_2$ are continuous, then $U$ is open and we can assume that it is connected (otherwise we repeat the argument for each connected component). Without loss of generality, we can assume that $A(\nabla u_1) \in W^{1,p}(\Omega)$ and we define $E_1=\{ x \in \Om :\ \nabla u_1 \not \in K \}$.

Let $\phi = (u_1-u_2)_+$. Since $u_1 \leq u_2$ on $\pa \Om$, then $\phi \in W_0^{1,\infty}(\Om)$ and \eqref{Euler eq 2} yields:
\begin{equation*}
\int_U A(\nabla u_j) \cdot \nabla (u_1-u_2) dx = \int_U g (u_1-u_2) dx,\quad j=1,2.
\end{equation*}
By subtracting the two identities, we have
\begin{equation}\label{eq1 thcomp}
\int_U \left[ A(\nabla u_1) - A(\nabla u_2) \right] \cdot (\nabla u_1- \nabla u_2) dx = 0.
\end{equation}
We notice that Lemma \ref{lemma 1} yields $|\{\nabla u_1 \in K\}|=0$ and thus
\begin{multline*}
\int_U \left[ A(\nabla u_1) - A(\nabla u_2) \right] \cdot (\nabla u_1- \nabla u_2) dx = \\ = \int_{U\cap E_1 } \left[ A(\nabla u_1) - A(\nabla u_2) \right] \cdot (\nabla u_1- \nabla u_2) dx;
\end{multline*}
\eqref{eq1 thcomp} and the monotonicity condition in \eqref{A def} imply that
\begin{equation} \label{eq2 th comp}
\nabla u_1=\nabla u_2 \textmd{ a.e. in } U \cap E_1.
\end{equation}
Since $|\{\nabla u_1 \in K\}|=0$, we obtain that $\nabla u_1=\nabla u_2$ a.e. in $U$. Being $u_1=u_2$ on $\pa U$, we have that $u_1=u_2$ in $U$, which gives a contradiction. \qed
\end{proof}

It is clear that Theorem \ref{th comp} implies the uniqueness of a solution for \eqref{Euler eq 1}. Moreover, from Theorem \ref{th comp}, we also obtain the following comparison principle.

\begin{corollary}
Let $u_j,\: j=1,2,\ A$ and $g$ be as in Theorem \ref{th comp}. If $u_1 < u_2$ on $\pa \Om$ then $u_1 < u_2$ in $\overline{\Om}$.
\end{corollary}

\begin{proof}
Since $\pa \Omega$ is compact and $u_1$ and $u_2$ are continuous in $\overline{\Omega}$, there exists a constant $c>0$ such that $u_1+c\leq u_2$ on $\pa \Om$. Being $u_1+c$ a solution of \eqref{Euler eq 2}, Theorem \ref{th comp} yields $u_1+c \leq u_2$ in $\overline{\Om}$ and, since $c$ is positive, we conclude. \qed
\end{proof}

\end{document}